\documentclass{amsart}
\usepackage{tikz}
\newtheorem{theorem}{Theorem}[section]
\newtheorem{lemma}[theorem]{Lemma}
\newtheorem{corollary}[theorem]{Corollary}
\theoremstyle{definition}
\newtheorem{definition}[theorem]{Definition}
\newtheorem{example}[theorem]{Example}

\newtheorem{problem}[theorem]{Problem}

\usepackage{comment}
\usepackage{hyperref}

\newcommand{\mattwo}[4]{\left[ \begin{array} {cc} #1 & #2 \\ #3 & #4 \end{array} \right]}

\numberwithin{equation}{section}

\begin{document}

\title{On Two Orderings of Lattice Paths}


\author{PJ Apruzzese}
\address{University of Connecticut}
\curraddr{}
\email{pj.apruzzese@uconn.edu}
\thanks{The first author was supported by the NSF grants DMS-2054561 and DMS-1802067}

\author{Kevin Cong}
\address{Harvard University, Cambridge, MA 02138}
\curraddr{}
\email{kcong@college.harvard.edu}
\thanks{The second author was supported by Jane Street Capital, the National Security Agency, and the NSF Grants DMS-2052036 and DMS-2140043.}



\date{}

\dedicatory{}

\begin{abstract} The \emph{Markov numbers} are positive integers appearing as solutions to the Diophantine equation $x^2 + y^2 + z^2 = 3xyz$. These numbers are very well-studied and have many combinatorial properties, as well as being the source of the long-standing unicity conjecture. In 2018, \c{C}anak\c{c}{\i} and Schiffler showed that the Markov number $m_{\frac{a}{b}}$ is the number of perfect matchings of a certain snake graph corresponding to the Christoffel path from $(0,0)$ to $(a,b)$. Based on this correspondence, Schiffler in 2023 introduced two orderings on lattice paths. For any path $\omega$, associate a snake graph $\mathcal{G}(\omega)$ and a continued fraction $g(\omega)$. The ordering $<_M$ is given by the number of perfect matchings on $\mathcal{G}(\omega)$, and the ordering $<_L$ is given by the Lagrange number of $g(\omega)$. 

In this work, we settle two conjectures of Schiffler. First, we show that the path $\omega(a,b) = RR\cdots R UU \cdots U$ is the unique maximum over all lattice paths from $(0,0)$ to $(a,b)$ with respect to both orderings $<_M$ and $<_L$. We then use this result to prove that $\sup L(\omega)$ over all lattice paths is exactly $1+\sqrt5$. 

\end{abstract}

\maketitle

\section{Introduction}

In 1879, Andrey Markov \cite{M} initiated the study of the \textit{Markov triples}, that is, the solutions to the Diophantine equation $x^2+y^2+z^2=3xyz$. A \textit{Markov triple} is any such integer solution, and a \textit{Markov number} is any number which appears as the maximum number of some Markov triple. The smallest Markov numbers are $1$, $2$, $5$, $13$, and their corresponding Markov triples are $(1,1,1)$, $(1,1,2)$, $(1,2,5)$, $(1,5,13)$. 

The Markov numbers exhibit many combinatorial properties. For instance, they can be given a tree structure equivalent to that of the Farey or Stern-Brocot trees \cite{A}. Under this structure, if $\mathcal{M}$ denotes the set of all Markov numbers, there exists a correspondence $\phi: \mathbb{Q} \rightarrow M$ between Markov numbers and reduced fractions $\frac{p}{q}$. We denote the Markov number corresponding to a given fraction as $\phi\left(\frac{p}{q}\right) = m_\frac{p}{q}$. 

In his seminal paper, Markov showed a remarkable connection between Markov numbers and the theory of rational approximation. In particular, for any irrational number $\alpha$, its \textit{Lagrange number} $L(\alpha)$ is defined by $$L(\alpha) = \sup \left\{L: \left|\frac{p}{q} - \alpha\right| < \frac{1}{Lq^2} \mathrm{\: for\: infinitely \: many \: reduced \: } \frac{p}{q}\right\}.$$ The \textit{Lagrange spectrum} is the set of all possible values of $L(\alpha)$ for irrational $\alpha$. Then there is the following bijection between Markov numbers and the Lagrange numbers less than $3$. 

\begin{theorem}\cite{M}
    Let $m_n$ denote the $n$-th smallest Markov number and $L_n$ denote the $n$-th smallest Lagrange number. Then $L_n = \sqrt{9 - \frac{4}{m_n^2}}$.
\end{theorem}

In a different direction, Frobenius \cite{F} asked in 1913 whether every Markov number occurred as the largest number of a unique Markov triple. This question, known as the famous \textit{unicity conjecture}, has largely resisted attack and remained unsolved for over 100 years. In 2023, Lee, Li, Rabideau, and Schiffler \cite{LLRS} proved strict inequalities on certain pairs of Markov numbers using cluster algebras, which in turn implies unicity within lines with given slope. In particular, they showed the following.

\begin{theorem}\cite{LLRS}
    Let $p > p'$. Then:
    
    \begin{enumerate}
        \item If $\frac{p-p'}{q-q'} \geq -\frac{8}{7}$, $m_{\frac{p}{q}} > m_{\frac{p'}{q'}}$.
        \item If $\frac{p-p'}{q-q'} \leq -\frac{5}{4}$, $m_{\frac{p}{q}} < m_{\frac{p'}{q'}}$.
    \end{enumerate}
    
\noindent In particular, if $m_{\frac{p}{q}} = m_{\frac{p'}{q'}}$, then $-\frac{5}{4} < \frac{p-p'}{q-q'} < -\frac{8}{7}$. 
\end{theorem}

In 2018, \.{I}lke  \c{C}anak\c{c}{\i} and Ralf Schiffler \cite{CS} showed that the Markov numbers appear combinatorially as the number of perfect matchings of certain snake graphs. In particular, let $a > b$ be relatively prime positive integers. Let $\mathcal{P}(a,b)$ be the set of all lattice paths from $(0,0)$ to $(a,b)$, and let $\mathcal{D}(a,b) \subseteq \mathcal{P}(a,b)$ be the set of such paths lying entirely below the diagonal $y = \frac{b}{a}x$. Per convention, we will represent any path $\omega \in \mathcal{D}(a,b)$ via a sequence of $R$s and $U$s, with $R$ representing a rightward move and $U$ representing an upward move. Then the \emph{Christoffel path} $\omega_{a,b} \in \mathcal{D}(a,b)$ is the path closest to the diagonal; that is, the path obtained by starting at $(0,0)$, going up whenever possible, and going right otherwise. 

Now, let $\omega \in \mathcal{D}(a,b)$ be any lattice path below the diagonal. We define a snake graph $\mathcal{G}(\omega)$ as follows. Call a point $(i,j)$ a \emph{half-lattice} point if $2i, 2j \in \mathbb{Z}$. Then for every half-lattice point $(i,j) \in \omega$ with $i \geq 1, j \leq b-1$, place a square tile of side length $\frac{1}{2}$ with vertices $(i,j)$, $(i, j+\frac{1}{2})$, $(i-\frac{1}{2}, j)$, and $(i-\frac{1}{2}, j+\frac{1}{2})$. Then $\mathcal{G}(\omega)$ is the graph with  vertices corresponding to the vertices of some tile, and edges corresponding to edges of some tile. For instance, the Christoffel path $\omega_{3,2}$ and its corresponding snake graph $\mathcal{G}(\omega_{3,2})$ are shown in Figure \ref{1}.

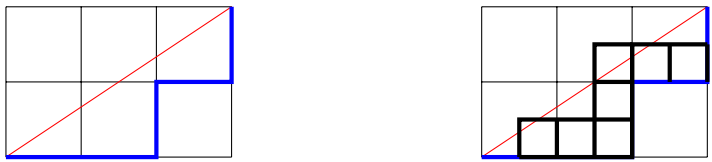
\begin{figure}[!ht]
\centering
\begin{minipage}{.5\textwidth}
  \centering
    \begin{tikzpicture}
        \draw[step=1.0,black,thin] (0,0) grid (3,2);
        \draw[red, thin] (0,0) -- (3,2);
        \draw[blue, ultra thick] (0,0) -- (2,0) -- (2,1) -- (3,1) -- (3,2);

    \end{tikzpicture}
  \label{fig:test1}
\end{minipage}%
\begin{minipage}{.5\textwidth}
  \centering
    \begin{tikzpicture}
        \draw[step=1.0,black,thin] (0,0) grid (3,2);
        \draw[red, thin] (0,0) -- (3,2);
        \draw[blue, ultra thick] (0,0) -- (2,0) -- (2,1) -- (3,1) -- (3,2);
        \draw[black, ultra thick] (0.5, 0) -- (2,0) -- (2, 1.5) -- (3, 1.5);
        \draw[black, ultra thick] (0.5, 0) -- (0.5, 0.5) -- (2, 0.5);
        \draw[black, ultra thick] (1, 0) -- (1, 0.5);
        \draw[black, ultra thick] (1.5, 0) -- (1.5, 0.5);
        \draw[black, ultra thick] (2, 0) -- (2, 1.5);
        \draw[black, ultra thick] (1.5, 0.5) -- (1.5, 1.5) -- (2, 1.5);
        \draw[black, ultra thick] (1.5, 1) -- (2, 1);
        \draw[black, ultra thick] (2.5, 1) -- (2.5, 1.5);
        \draw[black, ultra thick] (3, 1) -- (3, 1.5);

    \end{tikzpicture}
  \label{fig:test2}
\end{minipage}
\caption{The Christoffel path $\omega_{3,2} = RRURU$ (left) and its snake graph $\mathcal{G}(\omega_{3,2})$ (right). The corresponding continued fraction is $f(\omega_{3,2}) = [1,1,2,2,2,2]$.}
\label{1}
\end{figure}

Lastly, for any $\omega \in \mathcal{D}(a,b)$, recall that by convention we write $\omega$ as a word $W = w_1w_2 \cdots w_{a+b}$ of $R$ and $U$ moves. Construct a sequence as follows. For every $i$, if $w_i = w_{i+1}$, append $1,1$ to the sequence. If $w_i \neq w_{i+1}$, append $2$ to the sequence. This yields some sequence $(a_0, a_1, \ldots, a_t)$ of $1$s and $2$s. Denote by $f(\omega) = [a_0, a_1, \ldots, a_t]$ the continued fraction associated to the sequence. Then the following equivalence holds.

\begin{theorem}\cite{CS}
    The number of perfect matchings of the graph $\mathcal{G}(\omega)$ is equal to the numerator $\mathcal{N}$ of the continued fraction $f(\omega)$. Furthermore, the Markov number $m_{\frac{a}{b}}$ is precisely the number of perfect matchings of the graph $\mathcal{G}(\omega_{a,b})$, hence $m_{\frac{a}{b}} = \mathcal{N}$.
\end{theorem}

Generalizing this correspondence, Schiffler \cite{S} defined and initiated the study of two order relations on lattice paths in $\mathcal{D}(a,b)$. For any lattice path $\omega \in \mathcal{D}(a,b)$, let $M(\omega)$ be the number of perfect matchings of $\mathcal{G}(\omega)$. Suppose that $f(\omega) = [a_0,a_1, \ldots, a_t]$. Then let $L(\omega)$ be the Lagrange number of the infinite periodic continued fraction $g(\omega) = [\overline{2,a_0,a_1, a_2, \ldots, a_t}]$. We define two order relations based on the functions $M$ and $L$. 

\begin{definition} Let $\omega, \omega' \in \mathcal{D}(a,b)$. Then: 

    \begin{enumerate}
        \item The \textit{matching ordering} relation $<_{M}$ is given by $\omega <_M \omega'$ iff $M(\omega) < M(\omega')$.
        \item The \textit{Lagrange ordering} relation $<_L$ is given by $\omega <_L \omega'$ iff $L(\omega) < L(\omega')$. 
    \end{enumerate}
\end{definition}

These orderings induce posets on $\mathcal{D}(a,b)$. Our main result completely classifies the maximal element of $\mathcal{D}(a,b)$ with respect to both orders, proving a conjecture of Schiffler \cite{S}. 
\begin{theorem}\label{6.1}
    The element $\omega(a,b) = \overbrace{RR\cdots R}^a\overbrace{UU \cdots U}^b$ is the unique maximal element of $\mathcal{D}(a,b)$ as a poset with respect to both $<_M$ and $<_L$.
\end{theorem}

Let $\mathcal{D} = \bigcup_{(a,b)} \mathcal{D}(a,b)$ be the set of all lattice paths (of any length) below the diagonal. We then calculate exactly the supremum of the values $L(w)$ for $w \in \mathcal{D}$, proving another conjecture of Schiffler \cite{S}.

\begin{theorem}\label{6.6}
    We have $\sup \{L(\omega): \omega \in \mathcal{D}\} = 1+\sqrt5$. 
\end{theorem}

The remainder of this paper is organized as follows. In Section \ref{2}, we present preliminary definitions and lemmas to be used in the proofs of the main theorems. In Section \ref{3}, we prove Theorem \ref{6.1} for the ordering $<_M$, and in Section \ref{4}, we prove Theorem \ref{6.1} for the ordering $<_L$. In Section \ref{5}, we prove Theorem \ref{6.6}. Lastly, in Section \ref{6}, we give concluding remarks and suggest possible further work and open problems.

\textbf{Acknowledgements.} This research was conducted while the second author was at the 2023 University of Minnesota Duluth REU. The second author thanks Joe Gallian and Colin Defant for the opportunity to take part. The authors also thank Ralf Schiffler and Hanna Mularczyk for helpful comments, and Joe Gallian and Colind Defant for proofreading the paper.

\section{Preliminaries: Continued Fractions and Lagrange Numbers}\label{2}

We now present several important definitions and lemmas which will be repeatedly used. First, let us note the following continued fraction results from Aigner \cite{A}. The first is a classical result concerning continued fraction convergents which will be repeatedly used in calculating the values of relevant continued fractions. A proof is given in Aigner; additionally, we set a convention to resolve $i = 0$. 

\begin{lemma}\cite{A}\label{repeat}
    Let $\alpha = [a_0, a_1, \ldots]$ be a continued fraction and define its \emph{convergents} to be $\frac{p_i}{q_i} = [a_0, a_1, \ldots,a_i]$, where $p_{-1} = 1, q_{-1} = 0$ by convention. Let $\alpha_k = [a_k, a_{k+1}, \ldots]$. Then for all $i \geq 0$ we have $\alpha_i = \frac{p_i\alpha_{i+1} + p_{i-1}}{q_i\alpha_{i+1} + q_{i-1}}$.
\end{lemma}

For some proofs, it will also be convenient to represent continued fractions via matrices. 

\begin{definition}
    Let $\alpha = [a_0, a_1, \ldots, a_k]$ be a continued fraction and let $\frac{p_i}{q_i}$ be its convergents. Then the \emph{matrix representation} of $\alpha$ is $\mu(\alpha) = \mattwo{p_k}{p_{k-1}}{q_k}{q_{k-1}}$.
\end{definition}

The following is the analog of Lemma \ref{repeat} for matrix representations.

\begin{lemma}
    Let $\alpha = [a_0, a_1, \ldots, a_k]$. Then we have $$\mu(\alpha) = \mattwo{a_0}{1}{1}{0}\mattwo{a_1}{1}{1}{0} \cdots \mattwo{a_k}{1}{1}{0}.$$
\end{lemma}
As a corollary, we can inductively calculate the value of the following continued fractions. 

\begin{corollary}\label{1s}
    Let $F_t$ denote the $t$-th Fibonacci number. The following identities hold:
    \begin{enumerate}
        \item $[\overbrace{1,1,\ldots,1}^{t}] = \frac{F_{t+1}}{F_t}$,
        \item $[2,\overbrace{1,1,\ldots,1}^t] = \frac{F_{t+3}}{F_{t+1}}$, and
        \item $[\overbrace{1,1,\ldots,1}^s, 2, \overbrace{1,1,\ldots,1}^t] = \frac{F_{s+1}F_{t+3} + F_sF_{t+1}}{F_sF_{t+3}+F_{s-1}F_{t+1}}$.
    \end{enumerate}
\end{corollary}

The next lemma is used to compare values of continued fractions with many equal convergents. For this lemma, we set a convention that any finite continued fraction $\alpha = [a_0, \ldots, a_k]$ can be represented as an infinite continued fraction via $\alpha = [a_0, \ldots, a_k, \infty, \infty, \ldots]$. 

\begin{lemma}\label{useful}\cite[Lemma 1.24]{A}
    Let $\alpha = [a_0, \ldots]$ and $\beta = [b_0, \ldots]$. Let $k$ be the first index where $a_k \neq b_k$. Then:
    \begin{enumerate}
        \item $\alpha < \beta$ if and only if $(-1)^k a_k < (-1)^kb_k$, and
        \item $|\alpha - \beta| \leq \frac{1}{2^{k-2}}$ (if $k \geq 1$). 
    \end{enumerate}
\end{lemma}

Lastly, we note the following alternate expression for Lagrange numbers in terms of a maximum of certain sums of continued fractions, which follows from two results in \cite{A}. We will mainly work with this expression. 


\begin{lemma}\label{lagrange_max}\cite[Lemma 1.28, Proposition 1.29]{A}
    Let $\alpha = \left[ \overline{a_0, a_1, \dots, a_n} \right]$ be an infinite periodic continued fraction and define its \emph{shifts} to be $\rho_k = \left[ \overline{a_k, a_{k+1}, \dots, a_{k-1}} \right]$ for $0 \leq k \leq n$. Similarly, define its \emph{conjugate shifts}  $\rho_k' = -\left[0, \overline{a_{k-1}, a_{k-2}, \ldots, a_k}\right]$. Then
    \[ L(\alpha) = \max_{0 \leq k \leq n} (\rho_k - \rho_k'). \]
\end{lemma}

\section{Proof of Theorem \ref{6.1} for $<_M$}\label{3}
We begin with the following general statement, which exactly calculates the difference between the numerators of certain continued fractions.
\begin{lemma}\label{general_6.1.1}
    Let $$\alpha = [ a_0, \dots, a_k, A, c_0, \dots, c_m, B, b_0, \dots, b_{\ell} ]$$ be any continued fraction such that both $A \geq 2$ and $B \geq 2$, and let $$\alpha' = [ a_0, \dots, a_k, A - 1, 1, c_m, \dots, c_0, 1, B - 1, b_0, \dots, b_{\ell} ].$$ Denote the $i$-th convergent of the continued fractions $[a_0, \dots, a_k]$, $[b_0, \dots, b_{\ell}]$, and \newline $[c_0, \dots, c_m]$ by $\frac{p_i}{q_i}$, $\frac{r_i}{s_i}$, and $\frac{u_i}{v_i}$, respectively. Then the numerator of $\alpha'$ is greater than or equal to that of $\alpha$, and their difference is given by
    \begin{align*}
        \mathcal{N} (\alpha') - \mathcal{N} (\alpha) &= ( u_{m-1} + v_m + v_{m-1} ) \\ &\hspace{0.47cm}( p_k r_{\ell} (AB - A - B) + p_k s_{\ell} (A - 1) + p_{k-1} r_{\ell} (B - 1) + p_{k-1} s_{\ell} ) .
    \end{align*}
\end{lemma}

\begin{proof}
    The matrix representation of the continued fraction $\alpha$ gives us
    \[ \mu(\alpha) = \mattwo{p_k}{p_{k-1}}{q_k}{q_{k-1}} \mattwo{A}{1}{1}{0} \mattwo{u_m}{u_{m-1}}{v_m}{v_{m-1}} \mattwo{B}{1}{1}{0} \mattwo{r_{\ell}}{r_{\ell-1}}{s_{\ell}}{s_{\ell-1}} ,\]
    and likewise
    \[ \mu(\alpha') = \mattwo{p_k}{p_{k-1}}{q_k}{q_{k-1}} \mattwo{A}{A-1}{1}{1} \mattwo{u_m}{v_m}{u_{m-1}}{v_{m-1}} \mattwo{B}{1}{B-1}{1} \mattwo{r_{\ell}}{r_{\ell-1}}{s_{\ell}}{s_{\ell-1}} .\]

    Multiplying out these matrices and taking the upper-left entries of the products gives us the numerators of $\alpha$ and $\alpha'$, and thus we see that
    \begin{align*}
        \mathcal{N} (\alpha) = \hspace{0.1cm} & p_k r_{\ell} (A B u_m + A u_{m-1} + B v_m + v_{m-1}) + {} \\
        & p_k s_{\ell} (A u_m + v_m) + {} \\
        & p_{k-1} r_{\ell} (B u_m + u_{m-1}) + {} \\
        & p_{k-1} s_{\ell} (u_m)
    \end{align*}
    and
    \begin{align*}
        \mathcal{N} (\alpha') = \hspace{0.1cm} & p_k r_{\ell} (A B u_m + (A-1) B u_{m-1} + A (B-1) v_m + (A-1) (B-1) v_{m-1}) + {} \\
        & p_k s_{\ell} (A u_m + (A-1) u_{m-1} + A v_m + (A-1) v_{m-1}) + {} \\
        & p_{k-1} r_{\ell} (B u_m + (B-1) u_{m-1} + B v_m + (B-1) v_{m-1}) + {} \\
        & p_{k-1} s_{\ell} (u_m + u_{m-1} + v_m + v_{m-1}) .
    \end{align*}

    Taking the difference $\mathcal{N} (\alpha') - \mathcal{N} (\alpha)$ will cancel the $u_m$ terms and allow $u_{m-1} + v_m + v_{m-1}$ to be factored out of the expression, yielding
    \begin{align*}
        \mathcal{N} (\alpha') - \mathcal{N} (\alpha) &= ( u_{m-1} + v_m + v_{m-1} ) \\ &\hspace{0.47cm}( p_k r_{\ell} (AB - A - B) + p_k s_{\ell} (A - 1) + p_{k-1} r_{\ell} (B - 1) + p_{k-1} s_{\ell} ),
    \end{align*} as desired. Furthermore, all the coefficients in the expression are nonnegative since $A \geq 2$ and $B \geq 2$, and therefore $\mathcal{N} (\alpha') \geq \mathcal{N} (\alpha)$.
\end{proof}
We now prove that $\omega(a,b)$ is the unique maximum of $\mathcal{D}(a,b)$ under $<_M$. First, we show the following lemma, which will specify a way to lower any path $\omega$ towards $\omega(a,b)$ without decreasing the value of $M(\omega)$. We will need the following definition.

\begin{definition}
    For a lattice path $\omega \in \mathcal{D}$, let a \emph{block} of $\omega$ be a maximal contiguous subsequence of $R$s or $U$s. The number of blocks of $\omega$ is then denoted by $b(\omega)$.
\end{definition}

The following special case of Lemma \ref{general_6.1.1} will be used in the main proof. 
\begin{lemma}\label{6.1.1}
    Let $\omega$ be any path of the form $$\omega = [S_1] R\overbrace{UU\cdots U}^u \overbrace{RR \cdots R}^rU [S_2],$$ where $[S_1]$ and $[S_2]$ are subpaths and $u,r \geq 1$. Let $\omega'$ be the path obtained by exchanging the middle blocks of $U$s and $R$s, that is, $$\omega' = [S_1]\overbrace{RR \cdots R}^{r+1} \overbrace{UU\cdots U}^{u+1}[ S_2].$$ Then $M(\omega') \geq M(\omega)$, and equality holds only if $[S_1]$ is empty or the singleton $U$. 
    
\end{lemma}
\begin{proof}
    First, by the construction of $f(\omega)$, there exist two sequences $(a_0, \ldots, a_k)$ and $(b_0, b_1, \ldots, b_l)$ of $1$s and $2$s such that $$f(\omega) = [a_0, \ldots, a_k, 2,\overbrace{1,1,\ldots,1}^{2u-2}, 2, \overbrace{1,1,\ldots, 1}^{2r-2}, 2,b_0, \ldots, b_l]$$ and $$f(\omega') = [a_0, \ldots, a_k, \overbrace{1,1,\ldots,1}^{2r}, 2, \overbrace{1,1,\ldots, 1}^{2u}, b_0, \ldots, b_l].$$

    Then applying Lemma \ref{general_6.1.1} with $A = B = 2$ and $(c_0, \ldots, c_m) = (\overbrace{1,\cdots,1}^{2u-2},2,\overbrace{1\cdots, 1}^{2r-2})$ and using Corollary \ref{1s}, we obtain that  \begin{align*}\label{w'>w} \frac{M(\omega') - M(\omega)}{p_k s_l + p_{k-1} r_l + p_{k-1} s_l} &= F_{2u-1}F_{2r} + F_{2u-2}F_{2r-2} + F_{2u-2}F_{2r+1} +\\ &\hspace{0.47cm} F_{2u-3}F_{2r-1} + F_{2u-2}F_{2r}+ F_{2u-3}F_{2r-2} \\ &= F_{2u}F_{2r} + F_{2u-1}F_{2r-2} + F_{2u-3}F_{2r-1} + F_{2u-2}F_{2r+1} \\ &= F_{2u}F_{2r} + F_{2u-1}F_{2r-2} + F_{2u-3}F_{2r-1} + F_{2u-2}(F_{2r-1} + F_{2r}) \\ &= F_{2u}F_{2r} + F_{2u-1}F_{2r-2} + F_{2u-1}F_{2r-1} + F_{2u-2}F_{2r} \\ &= 2F_{2u}F_{2r}.\end{align*}

    It follows that \begin{equation}\label{w'>w}M(\omega') - M(\omega) = 2F_{2u}F_{2r}[p_k s_l + p_{k-1} r_l + p_{k-1} s_l] \geq 0.\end{equation} This establishes the inequality. Lastly, note that $p_k \geq 1$ and $r_l \geq 1$ by definition. Therefore, equality holds in \ref{w'>w} only if $s_l = p_{k-1} = 0$. But if $p_{k-1} = 0$ then since $a_i \in \{1,2\}$ we must have $k = 0$, that is, $S_1$ has length $1$ or is empty. In addition, if $[S_1] = R$, then $k = 1 \neq 0$ and $a_0 = a_1 = 1$. It follows that equality holds only if $[S_1]$ is empty or the singleton $U$. This completes the proof.
\end{proof}


We can now prove the first part of the main theorem.

\begin{proof}[Proof of Theorem \ref{6.1} for $<_M$] We claim that for any $\omega \in \mathcal{D}(a,b)$ such that $\omega \neq \omega(a,b)$, $M(\omega) < M(\omega(a,b))$. Induct on the number of blocks of $\omega$, $b(\omega)$. Note that $\omega$ must begin with $R$ and end with $U$. Hence $b(\omega) = 2k$ must be even. 

Since $b(\omega) = 2$ holds only for $\omega = \omega(a,b)$, we first prove the base case of $b(\omega) = 4$. Let $$\omega = \overbrace{RR\cdots R}^{r_1} \overbrace{UU \cdots U}^{u_1} \overbrace{RR \cdots R}^{r_2} \overbrace{UU \cdots U}^{u_2}.$$ Recall that $a>b$ by assumption, and $\omega \in \mathcal{D}(a,b)$ lies entirely below the diagonal. Therefore, $r_1 \geq 2$. Let $[S_1] = \overbrace{RR\cdots R}^{r_1 - 1}$ and $[S_2] = \overbrace{UU\cdots U}^{u_2 - 1}$. Then by Lemma \ref{6.1.1}, it follows that since $$\omega(a,b) = [S_1]\overbrace{RR\cdots R}^{r_2 + 1}\overbrace{UU\cdots U}^{u_1 + 1}[S_2], $$ we have $M(\omega) < M(\omega(a,b))$. The inequality is strict as $[S_1]$ is nonempty and begins with $R$.

The proof of the inductive step is essentially identical. Suppose the statement holds for all paths $\rho$ such that $b(\rho) \leq 2(k-1)$. Let $\omega \in \mathcal{D}(a,b)$ satisfy $b(\omega) = 2k$. Then write $$\omega = [S_1]R\overbrace{UU \cdots U}^{u_1} \overbrace{RR \cdots R}^{r_1}U [S_2],$$ where $\overbrace{UU \cdots U}^{u_1} \overbrace{RR \cdots R}^{r_1}$ is some pair of consecutive blocks within $\omega$. By Lemma \ref{6.1.1}, if $$\omega' = [S_1]\overbrace{RR\cdots R}^{r_1 + 1} \overbrace{UU\cdots U}^{u_1+1} [S_2],$$ we have $M(\omega) \leq M(\omega')$. Furthermore, if the subpath $[S_1]R$ has $i$ blocks and $U[S_2]$ has $j$ blocks, then $\omega$ has precisely $i + j + 2 = 2k$ blocks, whereas $\omega'$ has $i + j = 2k-2$ blocks. By the inductive hypothesis, $M(\omega') < M(\omega(a,b))$. Hence, $M(\omega) < M(\omega(a,b))$. This completes the induction and the proof of the theorem. 
\end{proof}

\section{Proof of Theorem \ref{6.1} for $<_L$}\label{4}

We now prove that $\omega(a,b)$ is that maximal element of $\mathcal{D}(a,b)$ under the ordering $<_L$. First, we determine the form of $L(\omega)$ for any path $\omega \in \mathcal{D}(a,b)$. 
\begin{lemma}\label{split_at_2}
    Let $\alpha = \left[ \overline{a_0, a_1, \dots, a_n} \right]$ be any purely periodic continued fraction with $a_i \in \{1, 2\}$ for all $i$. If $\alpha \neq \left[ \overline{1} \right]$, then there exists $a_k = 2$ such that
    \[ L(\alpha) = 2 + \left[ 0, \overline{a_{k+1}, a_{k+2}, \dots, a_{k-1}, 2} \right] + \left[ 0, \overline{a_{k-1}, a_{k-2}, \dots, a_{k+1}, 2} \right]. \]
\end{lemma}
\begin{proof}
    By Lemma \ref{lagrange_max}, we have
    \begin{align*}
        L(\alpha) & = \max_{0 \leq k \leq n} (\rho_k - \rho_k') \\
        & = \max_{0 \leq k \leq n} \left( \left[ \overline{a_k, a_{k+1}, \dots, a_{k-1}} \right] + \left[ 0, \overline{a_{k-1}, a_{k-2}, \dots, a_k} \right] \right) \\
        & = \max_{0 \leq k \leq n} \left( a_k + \left[ 0, \overline{a_{k+1}, a_{k+2}, \dots, a_k} \right] + \left[ 0, \overline{a_{k-1}, a_{k-2}, \dots, a_k} \right] \right),
    \end{align*}
    where the maximum is taken over all shifts of $\alpha$. Suppose first that $a_k = 1$; note that $a_i \in \{1,2\}$ for all $i$. Write $$\left[0, \overline{a_{k+1}, a_{k+2}, \ldots, a_k}\right] = \left[0, \overline{c_1, c_2, \ldots, c_{2n}}\right]$$ for $c_i \in \{1,2\}$. By repeatedly using item (1) of Lemma \ref{useful}, we have that 
    \begin{align*}
     \left[0, \overline{c_1, c_2, \ldots, c_{2n}}\right]  &\leq \left[0, \overline{1, c_2, \ldots, c_{2n}}\right] \\
     &\leq \left[0, \overline{1,2,c_3, \ldots, c_{2n}}\right] \\
     & \hspace{0.2cm} \vdots \\
     & \leq \left[0, \overline{1,2,1,2,\cdots,1,2}\right] \\
     &= \left[0,\overline{1,2}\right].
    \end{align*}
     Similarly, $$\left[ 0, \overline{a_{k-1}, a_{k-2}, \dots, a_k} \right] \leq \left[0, \overline{1,2}\right].$$ Combining these inequalities and the aforementioned expression for $L(\alpha)$, we obtain
     \begin{align*}
         L_1(\alpha) & := \max_{\substack{0 \leq k \leq n \\ a_k = 1}} \left( a_k + \left[ 0, \overline{a_{k+1}, a_{k+2}, \dots, a_k} \right] + \left[ 0, \overline{a_{k-1}, a_{k-2}, \dots, a_k} \right] \right) \\        
         &\leq 1 + \left[ 0, \overline{1, 2} \right] + \left[ 0, \overline{1, 2} \right] \\
         &= 1 + \frac{2}{1 + \sqrt{3}} + \frac{2}{1 + \sqrt{3}} \\
         &= 2 \sqrt{3} - 1 < 2.5. 
    \end{align*}
    On the other hand, if $a_k = 2$, a similar argument shows that 
    \begin{align*}
       L_2(\alpha) & := \max_{\substack{0 \leq k \leq n \\ a_k = 2}} \left( a_k + \left[ 0, \overline{a_{k+1}, a_{k+2}, \dots, a_k} \right] + \left[ 0, \overline{a_{k-1}, a_{k-2}, \dots, a_k} \right] \right) \\   
        &\geq 2 + \left[ 0, \overline{2, 1} \right] + \left[ 0, \overline{2, 1} \right] \\
        &= \left[ \overline{1, 2} \right] + \left[ \overline{1, 2} \right] = \frac{1 + \sqrt{3}}{2} + \frac{1 + \sqrt{3}}{2} \\ 
        &= 1 + \sqrt{3} > 2.5.
    \end{align*}

Thus, we conclude that if $\alpha \neq \left[\overline{1}\right]$, the maximum in the expression for $L(\alpha)$ must be attained at some $k$ for which $a_k = 2$. Thus, we have 
\begin{align*}
    L(\alpha) = 2 + \left[ 0, \overline{a_{k+1}, a_{k+2}, \dots, a_k} \right] + \left[ 0, \overline{a_{k-1}, a_{k-2}, \dots, a_k} \right],
\end{align*} as desired.
\end{proof}

Now, we can prove the following strengthening of the main theorem for $<_L$.

\begin{theorem}\label{stronger}
    Suppose $\omega, \omega' \in \mathcal{D}(a,b)$ are two paths which correspond to the infinite periodic continued fractions $g(\omega)$ and $g(\omega')$, respectively. Further suppose that there are positive integers $k$ and $l$ such that there exists a shift $g(\omega')_{t'}$ of $g(\omega')$ beginning with $2$ such that \[ g(\omega')_{t'} = [\overline{2,\underbrace{1, 1, \dots, 1}_{2(k-1)}, 2, \ldots,2,\underbrace{1, 1, \dots, 1}_{2(l-1)}}], \] and such that $i < k$, $j < l$ or $i < l$, $j < k$ for all positive integers $i$ and $j$ for which there exists a shift $g(\omega)_t$ of $g(\omega)$ beginning with $2$ of the form \[ g(\omega)_t = [\overline{2,\underbrace{1, 1, \dots, 1}_{2(i-1)},2,\ldots,2, \underbrace{1, 1, \dots, 1}_{2(j-1)}}]. \] (In the above, the central omitted terms can be any sequence of 1s and 2s). Then we have $L(\omega) < L(\omega')$. 
\end{theorem}
\begin{proof}
    First by Lemma \ref{split_at_2} and the condition, we must have 
    \[ L(\omega) = 2 + [0, \overline{\underbrace{1, 1, \dots, 1}_{2(i-1)}, 2, \dots, 2, \underbrace{1, 1, \dots, 1}_{2(j-1)}, 2}] + [ 0, \overline{\underbrace{1, 1, \dots, 1}_{2(j-1)}, 2, \dots, 2, \underbrace{1, 1, \dots, 1}_{2(i-1)}, 2}], \]
    where  $i < k$, $ j < l$ or $i < l$, $ j < k$. Assume without loss of generality that $i < k$, $j < l$. Note that  by Lemma \ref{lagrange_max}, we also have  
    \begin{align*}
        L(\omega') & = \max_{0 \leq k \leq n} (\rho_k - \rho_k') \\
        & = \max_{0 \leq k \leq n} \left( a_k + \left[ 0, \overline{a_{k+1}, a_{k+2}, \dots, a_k} \right] + \left[ 0, \overline{a_{k-1}, a_{k-2}, \dots, a_k} \right] \right) \\ 
        &\geq 2 + [0, \overline{\underbrace{1, 1, \dots, 1}_{2(k-1)}, 2, \dots, 2, \underbrace{1, 1, \dots, 1}_{2(l-1)}, 2}] + [ 0, \overline{\underbrace{1, 1, \dots, 1}_{2(l-1)}, 2, \dots, 2, \underbrace{1, 1, \dots, 1}_{2(k-1)}, 2}].
    \end{align*}

    Let us consider the first continued fraction in the expression for $L(\omega)$ and the lower bound for $L(\omega')$ given above. Since they are equal for the first $2i-1$ terms and the $2i$-th term of the latter is smaller, by item (1) of Lemma \ref{useful} we have
    
    \[ [0, \overline{\underbrace{1, 1, \dots, 1}_{2(i-1)}, 2, \dots, 2, \underbrace{1, 1, \dots, 1}_{2(j-1)}, 2}] < [0, \overline{\underbrace{1, 1, \dots, 1}_{2(k-1)}, 2, \underbrace{1, 1, \dots, 1}_{2(l-1)}, 2}]. \]
    Similarly, we have
    
    \[ [ 0, \overline{\underbrace{1, 1, \dots, 1}_{2(j-1)}, 2, \dots, 2, \underbrace{1, 1, \dots, 1}_{2(i-1)}, 2}] < [0, \overline{\underbrace{1, 1, \dots, 1}_{2(l-1)}, 2, \underbrace{1, 1, \dots, 1}_{2(k-1)}, 2}]. \]
    
    It follows by adding the above inequalities that $L(\omega) < L(\omega')$, as desired.
\end{proof}

The main theorem for $<_L$ now follows as a corollary.

\begin{proof}[Proof of Theorem \ref{6.1} for $<_L$]
    
    First, note that the path $\omega(a,b)$ has the corresponding periodic continued fraction $$g(\omega(a,b)) = [\overline{2, \underbrace{1, 1, \dots, 1}_{2(a-1)}, 2, \underbrace{1, 1, \dots, 1}_{2(b-1)}}].$$ 
    
    On the other hand, let $\omega \neq \omega(a,b)$ be any other path in $\mathcal{D}(a,b)$. Then, because $\omega$ contains $a$ total $R$s and $b$ total $U$s and $\omega \neq \omega(a,b)$, any block of $R$s must have fewer than $a$ terms and likewise any block of $U$s must have fewer than $b$ terms. Then consider any shift of $g(\omega)$, say
    \[g(\omega)_t = [\overline{2,\underbrace{1, 1, \dots, 1}_{2(i-1)},2,\ldots,2, \underbrace{1, 1, \dots, 1}_{2(j-1)}}],\] which begins with $2$. Because $g(\omega)_t$ begins with $2$, it must correspond to some cyclic shift of $\omega$ which begins and ends with different moves. That is, $g(\omega)_t = g(\omega')$ for some $\omega' \in \mathcal{P}(a,b)$\footnote{Let us remark that $\omega'$ may not necessarily be in $\mathcal{D}(a,b)$. In this case, however, we define $f(\omega')$ and $g(\omega')$ in the same way as when $\omega' \in \mathcal{D}(a,b)$.} satisfying either $\omega' = R[X]U$ or $\omega' = U[X]R$. If $\omega'$ begins with $R$, it follows by the above remarks that $i < a$, $j < b$. Otherwise, we have $i < b$, $j < a$. In either case, using Theorem \ref{stronger} we find that $L(\omega) < L(\omega(a,b))$, completing the proof.
\end{proof}

Combining the results of these two sections yields a complete proof of Theorem \ref{6.1}. 
\section{Proof of Theorem \ref{6.6}}\label{5}

In this section, we use Theorem \ref{6.1} to prove Theorem \ref{6.6}. First, we prove the following lemma concerning the values of $L(\omega(a,b))$. 

\begin{lemma}\label{maincalcfor66}
    For all positive integers $a,b$ we have $$1+\sqrt5 - \left(\frac{1}{2^{2a-3}}+ \frac{1}{2^{2b-3}}\right)  < L(\omega(a,b)) < 1 + \sqrt5.$$
\end{lemma}
\begin{proof}
    As in the proof of Theorem \ref{6.1}, we know from Lemma \ref{split_at_2} that \[ L(\omega(a,b)) = 2 + [0, \overline{\underbrace{1, 1, \dots, 1}_{2(a-1)}, 2, \underbrace{1, 1, \dots, 1}_{2(b-1)}, 2}] + [0, \overline{\underbrace{1, 1. \dots, 1}_{2(b-1)}, 2, \underbrace{1, 1, \dots, 1}_{2(a-1)}, 2}]. \] Then by item $(1)$ of Lemma \ref{useful}, we know that $$L(\omega(a,b)) < 2 + [0,\overline{1, \ldots, 1}] + [0, \overline{1,\ldots,1}] = 1+\sqrt5$$ and furthermore by (2) of the same Lemma we have $$|1 + \sqrt5 - L(\omega(a,b))| < \frac{1}{2^{2a-3}} + \frac{1}{2^{2b-3}}.$$ This completes the proof. 
\end{proof}

Now, we proceed to the proof of the main result of this section. 

\begin{proof}[Proof of Theorem \ref{6.6}]
    Let $\omega \in \mathcal{D}(a,b)$ be some path and note that by Theorem \ref{6.1}, we have $L(\omega) \leq L(\omega(a,b))$. It follows from Lemma \ref{maincalcfor66} that $L(\omega) \leq L(\omega(a,b)) < 1 + \sqrt5$. 

    Now, take $a = n+1, b = n$. By Lemma \ref{maincalcfor66}, $L(\omega(n+1,n)) > 1 + \sqrt5 - \frac{5}{2^{2n-1}}$. Letting $n \rightarrow \infty$, we conclude that $$ \lim_{n \rightarrow \infty}  L(\omega(n+1,n)) = 1+\sqrt5.$$ Combining these two results, we conclude that $\sup_{\omega \in \mathcal{D}} L(\omega) = 1+\sqrt5$.
\end{proof}

\section{Concluding Remarks}\label{6}

We conclude with a few notes on other questions in \cite{S} and further work. In \cite{S}, Schiffler asked about the covering relation of $\mathcal{D}(a,b)$ with respect to $<_M$ and $<_L$. We reiterate this problem here. 

\begin{problem}[Schiffler, Problem 6.5]
   Classify the covering relation of $\mathcal{D}(a,b)$ with respect to either $<_M$ or $<_L$.
\end{problem}

While Lemma \ref{6.1.1} yields a partial characterization of the ordering $<_M$ and Theorem \ref{stronger} yields a method of comparing certain pairs of paths, a complete characterization of either $<_M$ or $<_L$ is far from complete. Nonetheless, it may be interesting to obtain a classification of which paths can be compared under either the orders $<_M$ or $<_L$ via our methods in this paper. 

\begin{problem}
    Give a combinatorial characterization all pairs of paths $(\omega, \omega')$ for which the method of Theorem \ref{stronger} yield a direct comparison of the two with respect to $<_M$ or $<_L$. 
\end{problem}

Lastly, Schiffler asked whether a relation between paths under either ordering could imply a relation under the other.

\begin{problem}[Schiffler, Problem 6.4]\label{6.4}
    Is it true that either: 
    \begin{itemize}
        \item $\omega <_M \omega' $ implies either $\omega <_L \omega'$ or $L(\omega) = L(\omega')$, or 
        \item $\omega <_L \omega' $ implies either $\omega <_M \omega'$ or $M(\omega) = M(\omega')$?
    \end{itemize}
\end{problem}

The following example shows this result in the negative. 

\begin{example}
    Consider the paths $\omega = RRRUURURU$ and $\omega' = RRRUURRUU$, as shown in Figure \ref{ex64}. We have $\omega, \omega' \in \mathcal{D}(5,4)$. Then $M(\omega) = 1115 < 1177 = M(\omega')$. However, we have $L(\omega) = \frac{\sqrt{11390621}}{1055} > \frac{17 \sqrt{48893}}{1177} = L(\omega')$. It follows that neither item of Problem \ref{6.4} can hold in general. Note that this example also shows that Lemma \ref{6.1.1} does not hold for the ordering $<_L$.
\end{example}

\begin{figure}[!ht]
\centering
\begin{minipage}{.5\textwidth}
  \centering
  \begin{tikzpicture}
        \draw[step=0.5,black,thin] (0,0) grid (2.5,2);
        \draw[red, thin] (0,0) -- (2.5,2);
        \draw[blue, ultra thick] (0,0) -- (1.5,0) -- (1.5,1) -- (2,1) -- (2,1.5) -- (2.5,1.5) -- (2.5,2);
    \end{tikzpicture}
  \label{fig:test1}
\end{minipage}%
\begin{minipage}{.5\textwidth}
  \centering
  \begin{tikzpicture}
        \draw[step=0.5,black,thin] (0,0) grid (2.5,2);
        \draw[red, thin] (0,0) -- (2.5,2);
        \draw[blue, ultra thick] (0,0) -- (1.5,0) -- (1.5,1) -- (2.5,1) -- (2.5,2);
    \end{tikzpicture}
  \label{fig:test2}
\end{minipage}
\caption{Lattice paths $\omega = RRRUURURU$  (left) and $\omega' = RRRUURRUU$ (right). We have $\omega <_M \omega'$ but $\omega' <_L \omega$.}
\label{ex64}
\end{figure}

However, the relation between the orders $<_M$ and $<_L$ remains an interesting question to consider. We therefore end with the following general problem.

\begin{problem}
    Classify all pairs $\omega, \omega'$ satisfying each of the following five classes of relations:
    \begin{enumerate}
        \item $\omega <_M \omega'$ and $\omega <_L \omega'$,
        \item $\omega <_M \omega' $ and $\omega' <_L \omega$,
        \item $\omega <_M \omega'$ and $L(\omega) = L(\omega')$,
        \item $M(\omega) = M(\omega')$ and $\omega <_L \omega'$, 
        \item $M(\omega) = M(\omega')$ and $L(\omega) = L(\omega')$. 
    \end{enumerate}

\end{problem}



\end{document}